\title{\textsc{\textbf{{
Spinors and mass on weighted manifolds}}}}
\author{\textsc{Julius Baldauf\thanks{Supported in part by the National Science Foundation. {\it E-mail}: \texttt{juliusbl@mit.edu}} 
\quad \quad \quad 
Tristan Ozuch}
\vspace{0.2cm}\\
    \textsc{\footnotesize MIT Department of Mathematics}\vspace{-0.1cm}\\
    \textsc{\footnotesize Cambridge, MA} \vspace{-0.05cm}
}
\date{}

\documentclass[11pt, letterpaper, leqno, twoside]{article}

\usepackage[colorinlistoftodos,prependcaption]{todonotes}

\usepackage{amsmath}
\usepackage{amsthm,xcolor}
\usepackage[colorlinks, linkcolor=red]{hyperref}
\usepackage[top=1.5in, lmargin=1in, rmargin=1in, marginparwidth=0.9in]{geometry}
\usepackage[english]{babel}
\usepackage{amssymb}
\usepackage{mathrsfs}
\usepackage{mathtools}
\usepackage{bbm}
\usepackage{amsthm}
\usepackage{fancyhdr}
\usepackage{tikz-cd}
\usepackage{comment}
\usepackage{etoolbox}
\usepackage{breqn}
\usepackage{esint}
\usepackage{appendix}
\hypersetup{
    colorlinks=true,
    linkcolor=blue,
    filecolor=magenta,      
    urlcolor=cyan,
	citecolor=blue
}

\usepackage[T1]{fontenc}
\usepackage[center]{titlesec}
\usepackage{xcolor}
\usepackage[bottom]{footmisc}
\usepackage{indentfirst}
\usepackage{xpatch}
\usepackage{chngcntr}

\makeatletter
\renewcommand\th@plain{\slshape}
\xpatchcmd{\proof}{\itshape}{\slshape}{}{}
\makeatother

\makeatletter
\renewcommand\th@plain{\slshape}
\makeatother

\titlelabel{\thetitle.\quad}
\titleformat*{\section}{\centering\large\scshape\sffamily}
\titleformat{\subsection}[runin]
  {\normalfont\bfseries}{\thesubsection.}{0.6em}{}
\titleformat{\subsubsection}[runin]
  {\normalfont\bfseries}{\thesubsubsection.}{0.6em}{}

\fancypagestyle{mypagestyle}{%
    \fancyhf{}
    
    \fancyhead[OC]{\small Julius Baldauf \qquad \qquad \qquad \qquad Tristan Ozuch}
    \fancyhead[EC]{\small Spinors and mass on weighted manifolds}
    \fancyhead[EL]{\thepage}
    \fancyhead[OR]{\thepage}
    \setlength{\headheight}{13.6pt}
}
\numberwithin{equation}{section}

\theoremstyle{plain} 
\newtheorem{lemma}[equation]{Lemma}

\newtheorem{proposition}[equation]{Proposition}
\newtheorem{theorem}[equation]{Theorem}
\newtheorem{corollary}[equation]{Corollary}

\theoremstyle{definition}

\newtheorem{remark}[equation]{Remark}

\newcommand{\R}{\mathbb{R}}
\newcommand{\Z}{\mathbb{Z}}
\newcommand{\N}{\mathbb{N}}

\newcommand{\id}{\mathbbm{1}}

\newcommand{\Div}{\mathrm{div}}
\newcommand{\be}{\begin{equation}}
\newcommand{\ee}{\end{equation}}

\newcommand{\Ric}{\mathrm{Ric}}

\newcommand{\Scal}{\operatorname{R}}

\renewcommand{\phi}{\varphi}

\newcommand{\Rea}{\,\mathrm{Re}\,}
\newcommand{\Ima}{\,\mathrm{Im}\,}
\newcommand{\tr}{\mathrm{tr}}
\newcommand{\Hess}{\mathrm{Hess}}
\newcommand{\mass}{\mathfrak{m}}
\newcommand{\intprod}{\lrcorner\,}
\newcommand{\euc}{\mathrm{euc}}

\newcommand{\dL}{\Delta}

\newcommand{\lambdaALE}{\lambda_{\mathrm{ALE}}}

\begin{document}
\maketitle
\begin{abstract}
This paper generalizes classical spin geometry to the setting of weighted manifolds (manifolds with density) and provides applications to the Ricci flow. 
Spectral properties of the naturally associated weighted Dirac operator, introduced by Perelman, and its relationship with the weighted scalar curvature are investigated.
Further, a generalization of the ADM mass for weighted asymptotically Euclidean (AE) manifolds is defined;
on manifolds with nonnegative weighted scalar curvature, it satisfies a weighted Witten formula and thereby a positive weighted mass theorem. Finally, on such manifolds, Ricci flow is the gradient flow of said weighted ADM mass, for a natural choice of weight function. 
This yields a monotonicity formula for the weighted spinorial Dirichlet energy of a weighted Witten 
spinor along Ricci flow.
\end{abstract}


\section{Introduction}

Manifolds \emph{with density}, or \emph{weighted} manifolds, have long appeared in mathematics. 
A weighted manifold is a Riemannian manifold $(M,g)$ endowed with a function $f:M\to \R$, defining the measure $e^{-f}dV_g$.
After being introduced by Lichnerowicz in \cite{lic1,lic2},  more recent attention has been given to the differential geometry of weighted manifolds, including a generalization of Ricci curvature.
A central idea of Perelman's spectacular proofs \cite{P} required considering manifolds with density and their evolution. 
This led him to introduce a notion of weighted scalar curvature which is \emph{not} the trace of the weighted Ricci curvature of Bakry-Émery.
Sometimes called the P-scalar curvature, this weighted scalar curvature has only been moderately studied; see for instance \cite{fan,ac,lm, D, BH}. 

This paper shows that the intimate relationship between scalar curvature and the Dirac operator generalizes naturally to the weighted scalar curvature and an associated weighted Dirac operator, defined below. 
Well-known theorems relating scalar curvature and the Dirac operator include Friedrich's eigenvalue estimate \cite{F1}, Witten's proof of the positive mass theorem \cite{W}, Gromov-Lawson's obstructions to positive scalar curvature \cite{GL}, and the Seiberg-Witten theory \cite{W3}. Here, the first two of said theorems are generalized and then applied to the Ricci flow.

Aside from their applications in Ricci flow, weighted manifolds have proven extremely useful in the context of diffusion operators in analysis and probability theory, starting with Bakry and Émery's celebrated article \cite{be}. In a more classical Riemannian geometry context, Cheeger-Colding showed that limits of collapsing manifolds are naturally endowed with densities. 
Such densities differ from those defined by the Riemannian volume form, and the natural object of study is a \emph{metric measure space}. See also the many extensions to the theory of (R)CD spaces started in \cite{lv,stu}.

In physics, manifolds with density appear in a number of theories arising from Kaluza-Klein compactifications, via the mechanism of dimensional reduction. The closest to the purpose of this paper is probably Brans-Dicke theory, which motivates the study of manifolds with density and Bakry-\'Emery's notion of (weighted) Ricci curvature in \cite{gw,ww,lmo}. 
Also, the weighted version of the Hilbert-Einstein action, introduced by Perelman, appears as the Lagrangian in several gravitational theories; this fact was noted in \cite{ccd}, for instance.

Table \ref{table:1} gives a summary comparison between classical and weighted quantities. 
The weighted quantities are typically better behaved than their Riemannian counterparts as one can choose a geometrically meaningful density. This idea can be seen as the core of Perelman's proofs \cite{P}. In the context of scalar curvature and mass questions, proofs often employ a conformal change of the metric to reach \emph{constant} scalar curvature, significantly changing the geometry; see \cite{CP} for a survey of this technique. In contrast, on weighted manifolds, the idea is rather to fix the background geometry while varying the weight in order to obtain a metric with \emph{constant} weighted scalar curvature.

\begin{table}[h!]
\begin{center}
\renewcommand{\arraystretch}{1.2} 
\begin{tabular}{ l|l|l } 
& Riemannian & with density\\
 \hline
  Volume form & $dV$ &$e^{-f}dV$ \\ 
  Ricci curvature & $\Ric$ & $\Ric_{f}:=\Ric +\Hess_f$ \\ 
  Scalar curvature & $\Scal$ & $\Scal_{f}:=\Scal +2\Delta f-|\nabla f|^2$ \\ 
 Hilbert-Einstein fct. &$\operatorname{HE}:=\int_M\Scal dV$&$\mathcal{F}(f):=\int_M\Scal_{f}e^{-f}dV$ \\
 Einstein's tensor & $\operatorname{E}:=\Ric-\frac{\Scal}{2}g$ & $\operatorname{E}_f:=\Ric_f-\frac{\Scal_f}{2}g$\\
 Divergence & $\Div$ & $\Div_{f}(h):=\Div(h) - h(\nabla f,\cdot)$\\
 Bianchi identity&$\Div (\operatorname{E})=0$&$\Div_{f}(\operatorname{E}_f)=0$ \\
 Einstein metric &$\Ric=\Lambda g$& $ \Ric_f=\Lambda g $\\
 Mean curvature & $H$ & $H_f:= H-\nabla_\nu f$\\ 
 Dirac operator* & $D$ & $D_f:=D-\frac{1}{2}(\nabla f)\cdot$ \\
 Lichnerowicz formula* &$D^2=-\Delta+\frac{1}{4}\Scal $&$D_f^2=-\Delta_f+\frac{1}{4}\Scal_f $\\
 Ricci identity* &$[D,\nabla_X]=\frac{1}{2}\Ric(X)\cdot $&$[D_f,\nabla_X]=\frac{1}{2}\Ric_f(X)\cdot $\\
 Dirac spinor* & $\psi$ s.t. $D\psi=0$ & $\psi_f:=e^{-\frac{f}{2}}\psi$ s.t. $D_f\psi_f=0$\\
 Eigenvalue bound* &$\lambda(D)^2\geqslant  \frac{n}{4(n-1)}\min\Scal$ & $\lambda(D)^2=\lambda(D_f)^2\geqslant \frac{n}{4(n-1)}\min\Scal_f$ \\
 ADM mass*& $\mass 
    =\lim\limits_{\rho\to \infty}\int_{S_{\rho}}(\partial_ig_{ij}-\partial_jg_{ii})\, dA_j$ & $\mass_f:=\mass+2\lim\limits_{\rho\to \infty}\int_{S_{\rho}}\langle \nabla f,\nu\rangle\,e^{-f}dA$\\
 Witten formula* & $\mass=4\int_M\left(|\nabla\psi|^2+\frac{1}{4}\Scal|\psi|^2\right) dV$ &$\mass_f=4\int_M\left(|\nabla\psi|^2+\frac{1}{4} \Scal_f|\psi|^2\right) e^{-f}dV$ \\
 \hline
\end{tabular}
\end{center}
\vspace{-0.4cm}
\caption{\small Classical vs. weighted quantities. Contributions from this paper are labeled with an asterisk (*).}
\label{table:1}
\end{table}
\vspace{-0.5cm}

\subsection{Weighted Dirac operator.}
Section \ref{sec: Weighted Dirac operator} extends classical spin geometry theory to weighted manifolds. The new mathematical object introduced in this section is the weighted Dirac operator, 
\begin{equation}
    D_f=D-\frac{1}{2}(\nabla f)\cdot.
\end{equation}
The $\nabla f$ term acts by Clifford multiplication, and $D$ denotes the standard (unweighted) Dirac operator. The weighted Dirac operator is self-adjoint with respect to the weighted $L^2$-inner product 
and is unitarily equivalent to the standard Dirac operator; see Proposition \ref{prop: unitary equivalence}. 

Differential operators naturally associated with weighted measures have proven invaluable in analysis and geometry. Of particular note is the weighted Laplacian, $\Delta_f=\Delta-\nabla_{\nabla f}$, also called the drift Laplacian, $f$-Laplacian, or Witten Laplacian. 
When $f=\frac{|x|^2}{4}$ on $\R^n$, then $\Delta_f$ is the Ornstein-Uhlenbeck operator. 
Weighted Laplace operators have been used 
in Ricci and mean curvature flow to analyze solitons \cite{CM, CZ, MW},
and by Witten in his study of Morse theory \cite{W2}, for example.

Proposition \ref{prop: weighted Lichnerowicz} proves a weighted Lichnerowicz formula involving the weighted scalar curvature,
\begin{equation}
    D_f^2=-\dL_f+\frac{1}{4}\Scal_f.
\end{equation}
Proposition \ref{prop: weighted Ricci identity} proves a weighted Ricci identity involving the Bakry-\'Emery Ricci curvature,
\begin{equation}
    [D_f,\nabla_X]=\frac{1}{2}\Ric_f(X)\cdot.
\end{equation}
Theorem \ref{thm: weighted friedrich} generalizes the classical lower bound for Dirac eigenvalues to the weighted setting: on a closed, weighted spin manifold, any eigenvalue $\lambda$ of $D_f$ satisfies
\begin{equation}
    \lambda^2 \geqslant \frac{n}{4(n-1)}\min \Scal_f. 
\end{equation}
Furthermore, the same lower bound also holds for eigenvalues of the standard Dirac operator.

Forthcoming work will study weighted spin manifolds with boundary \cite{BO2}.

\subsection{Weighted asymptotically Euclidean manifolds.}
A fundamental quantity associated with an asymptotically Euclidean (AE) manifold $(M^n,g)$ is the ADM mass \cite{ADM}, denoted $\mass(g)$.
Section \ref{sec: weighted AE manifolds} introduces a quantity extending the ADM mass to the weighted setting: 
the \emph{weighted mass} of an AE manifold with weight function $f$ is defined as
\begin{equation}
    \mass_f(g):= \mass(g) +2\lim_{\rho\to \infty}\int_{S_{\rho}}\langle \nabla f, \nu\rangle\,e^{-f}dA,
\end{equation}
where $S_{\rho}$ is a coordinate sphere of radius $\rho$ with outward normal $\nu$ and area form $dA$.
The normalization for $\mass$ used in this paper is related to Bartnik's \cite{Ba86} by $\mass=c_nm_{\mathrm{ADM}}$, where $c_n=2(n-1)\omega_{n-1}$ and $\omega_{n-1}$ is the area of the unit sphere in $\R^n$; this simplifies the formulas to follow.

Theorem \ref{thm: weighted witten} shows that the weighted mass of a spin manifold satisfies a weighted Witten formula: if the weighted scalar curvature is nonnegative and $f$ decays suitably rapidly at infinity, there exists an asymptotically constant weighted-harmonic spinor $\psi$ of norm 1 at infinity and satisfying
\begin{equation}\label{eqn: weighted Witten intro}
    \mass_f(g)
        =4\int_M\left(|\nabla \psi|^2+\frac{1}{4}\Scal_f|\psi|^2\right)e^{-f} dV_g.
\end{equation}
Moreover, Theorem \ref{thm: positive weighted mass} proves a positive weighted mass theorem on spin manifolds: if the weighted scalar curvature is nonnegative and $f$ decays suitably rapidly at infinity, then
\begin{equation}
    \mass_f(g)\geq 0, \quad \text{with equality iff $(M^n,g)\cong (\R^n,g_{\euc})$ and $\int_{\R^n}(\Delta_f f)\,e^{-f}dV_{g_{\euc}}=0$}.
\end{equation}

By way of a parenthetical remark: using work of Nakajima \cite{N} (see \cite{DO20}), the results of this section have straightforward extensions to asymptotically locally Euclidean spaces of dimension $4$ with subgroup $\mathrm{SU}(2)$ at infinity, though they are not pursued in this paper.

\subsubsection{Weighted mass and Ricci flow.}
ADM mass does not measure how far a manifold is from the Euclidean metric, except in an asymptotic way at infinity. Indeed, one striking way to see this is that 3-dimensional Ricci flow (with surgery) starting at an AE metric with nonnegative scalar curvature converges to Euclidean space \cite{Li18};
however, 
\emph{mass is constant} along the flow and thus does not detect the improvement of the geometry \cite{DM,OW, Ha2, Li18}. 

On the other hand, with a suitable choice of weight function $f$, the weighted mass indeed measures how far an AE manifold is from Euclidean space: the most natural choice for $f$ is the unique $f_g$ decaying at infinity and solving $\Scal_{f_g}\equiv 0$. Theorem \ref{thm: weighted mass = lambda} shows that such an $f_g$ exists on any AE manifold with nonnegative scalar curvature. This surprisingly yields the formula
\begin{equation}\label{eqn: weighted mass equals lambdaALE}
    \mass_{f_g}(g) 
    = -\lambdaALE(g), 
\end{equation}
where $\lambdaALE(g)$ is the renormalized Perelman functional introduced by Deruelle and the second author \cite{DO20}. 
Equality (\ref{eqn: weighted mass equals lambdaALE}) is the content of Theorem \ref{thm: weighted mass = lambda}, and is unexpected at first sight since $\lambdaALE$ stems from a variational principle on the whole manifold, and a priori is not a boundary term. (The notation for $\lambdaALE$ is adopted from \cite{DO20}, since the results here also apply to ALE spaces.)

The renormalized Perelman functional is the correct modification of Perelman's $\lambda$-functional (for closed manifolds) to AE manifolds: it has the crucial property that Ricci flow, $\partial_tg = -2\Ric$, is its gradient flow \cite{DO20, Ha}. Thus equality (\ref{eqn: weighted mass equals lambdaALE}) implies that a Ricci flow on an AE manifold with nonnegative scalar curvature is the gradient flow of the weighted mass (see Corollary \ref{cor: monotonicity of weighted mass}):
\begin{equation}\label{eqn: derivative of weighted mass}
    \frac{d}{dt}\mass_{f_{g}}(g)
    =-2\int_M|\Ric+\Hess_{f_g}|^2e^{-f_g} dV \leq 0,
\end{equation}
and equality implies Ricci-flatness.
Together, (\ref{eqn: weighted Witten intro}), (\ref{eqn: weighted mass equals lambdaALE}), and (\ref{eqn: derivative of weighted mass}) imply the following monotonicity formula along Ricci flow for the weighted \emph{spinorial} Dirichlet energy of a weighted Witten spinor:
\begin{equation}\label{eqn: Dirichlet monotonicity formula}
    \frac{d}{dt}\int_M|\nabla\psi|^2e^{-f_g}dV
    =-\frac{1}{2}\int_M|\Ric+\Hess_{f_g}|^2e^{-f_g} dV.
\end{equation}
This monotonicity formula stands in contrast to the constancy of ADM mass along Ricci flow, which implies that for an (unweighted) Witten spinor $\phi$, the integral
$
    \int_M(|\nabla \phi|^2 +\frac{1}{4}\Scal |\phi|^2)dV
$
is constant along Ricci flow.
Further applications of spin geometry to the Ricci flow, including a direct proof of (\ref{eqn: Dirichlet monotonicity formula}) via the first variation, will be presented in forthcoming work \cite{BO1}.

Equality (\ref{eqn: weighted mass equals lambdaALE}) additionally implies that all of the advantages of $\lambdaALE$ over the ADM mass also hold for the weighted mass. In addition to those already stated, the key advantages of the weighted mass over the ADM mass are as follows:
like ADM mass, $\mass_{f_g}(g)$ is nonnegative on any \emph{spin} AE manifold, and vanishes only on Euclidean space; 
$\mass_{f_g}(g)$ satisfies a \L{}ojasiewicz inequality measuring the distance to Euclidean space;
$\mass_{f_g}(g)$ is real-analytic on weighted H\"older spaces, where neither mass, nor the $L^1$-norm of scalar curvature are defined;
even when an AE manifold has some negative scalar curvature, $\mass_{f_g}(g)$ is nonnegative and detects how far from Euclidean space the geometry is, allowing for stability analysis of gravitational instantons under general perturbations \cite{DO21}.
\section*{\small\bf Acknowledgements}
The first author is indebted to William Minicozzi for continual support, and to Clifford Taubes for inspiring discussions. Part of this work was completed while the first author was funded by a National Science Foundation Graduate Research Fellowship.
\section{Weighted Dirac operator}\label{sec: Weighted Dirac operator}

\indent 
Let $(M^n,g)$ be a complete Riemannian spin $n$-manifold without boundary. The spin bundle $\Sigma M\to M$ is a complex vector bundle of rank $2^{\lfloor \frac{n}{2}\rfloor}$, equipped with a Hermitian metric, Clifford multiplication, and connection. These objects satisfy compatibility conditions which are stated below. A spinor field, or simply spinor, is a section of the bundle $\Sigma M$. For background on spin geometry, see the book \cite{BHMMM}, whose notation and conventions are adopted here.

Let $f\in C^{\infty}(M)$. The weighted Dirac operator $D_f:\Gamma(\Sigma M)\to \Gamma(\Sigma M)$ is defined as 
\begin{equation}
    D_f=D-\frac{1}{2}(\nabla f) \cdot \;,
\end{equation}
where $D=e_i\cdot \nabla_i$ is the standard (Atiyah-Singer) Dirac operator and $\cdot$ denotes Clifford multiplication. 
(Throughout this paper, 1-forms and vector fields will often be identified without explicit mention.)
The weighted Dirac operator is the Dirac operator associated with the modified spin connection $\nabla^f:\Gamma(\Sigma M)\to \Gamma(T^*M\otimes \Sigma M)$, defined by
\begin{equation}\label{eqn: weighted spin connection}
    \nabla_X^f\psi=\nabla_X \psi-\frac{1}{2}(\nabla_X f)\psi,
\end{equation}
where $\nabla$ is the standard spin connection induced by the Levi-Civita connection. The modified spin connection $\nabla^f$ is {\it not} metric compatible with the standard metric \cite[Prop. 2.5]{BHMMM} on the spin bundle, $\langle\cdot,\cdot\rangle$, however, it is compatible with the modified metric $\langle \cdot,\cdot\rangle_f := \langle \cdot,\cdot\rangle e^{-f}$, that is
\begin{equation}\label{eqn: weighted connection is weighted metric compatible}
    X(\langle \psi,\phi\rangle e^{-f})
    =\langle \nabla^f_X\psi,\phi\rangle e^{-f}+\langle \psi,\nabla^f_X\phi\rangle e^{-f},
\end{equation}
for any vector field $X$ and spinors $\psi,\phi$. Moreover, since Clifford multiplication is parallel with respect to the standard spin connection, it is also parallel with respect to $\nabla^f$. This means that
\begin{equation}\label{eqn: Cliff mult is weighted parallel}
    \nabla_X^f(Y\cdot \psi)=Y\cdot \nabla_X^f\psi +(\nabla_XY)\cdot \psi,
\end{equation}
for any vector fields $X,Y$ and spinor $\psi$.

The weighted Dirac operator satisfies the following weighted integration by parts formula on $W^{1,2}(e^{-f}\,dV)$,
\begin{equation}\label{eqn: weighted Laplcian IBP}
    \int_M\langle \psi,D_f\phi\rangle e^{-f}\,dV
    =\int_M \langle D_f \psi,\phi\rangle e^{-f}\,dV
\end{equation}
and hence is self-adjoint on $W^{1,2}(e^{-f}\,dV)$. 
Furthermore, a weighted Lichnerowicz formula holds, which was observed by Perelman \cite[Rem.\ 1.3]{P}. To state it, let
\begin{equation}
    \dL_f=\Delta  -\nabla_{\nabla f}
\end{equation} 
be the weighted Laplacian acting on spinors and let
\begin{equation}\label{eqn: defn of weighted scalar curvature}
    \Scal_f=\Scal+2\Delta f-|\nabla f|^2
\end{equation}
be Perelman's weighted scalar curvature (or P-scalar curvature).

\begin{proposition}
[Weighted Lichnerowicz]\label{prop: weighted Lichnerowicz}
The square of the weighted Dirac operator $D_f$ satisfies
\begin{equation}\label{eqn: weighted Lichnerowicz}
    D_f^2=-\dL_f+\frac{1}{4}\Scal_f.
\end{equation}
\end{proposition}

\begin{proof}
The proof is a consequence of the standard Lichnerowicz formula and the properties of Clifford multiplication. Recall that if $e_1,\dots,e_n$ is a local orthonormal basis of $TM$, then for any symmetric 2-tensor $A$,
\begin{equation}
    \sum_{i,j=1}^n A(e_i,e_j)e_i\cdot e_j\cdot=-\tr(A)\id.
\end{equation}
(The proof is immediate from the Clifford algebra relation $e_i\cdot e_j+e_j\cdot e_i=-2\delta_{ij}\id$). Combined with the standard Lichnerowicz formula and the Clifford algebra relation, it follows that for any smooth spinor $\psi$,
\begin{align}
    D_f^2\psi
    &=\left(D-\frac{1}{2}(\nabla f)\cdot\right)\left(D-\frac{1}{2}(\nabla f)\cdot\right) \psi \\
    &=D^2\psi -\frac{1}{2}D((\nabla f)\cdot \psi)-\frac{1}{2}(\nabla f)\cdot D\psi -\frac{1}{4}|\nabla f|^2\psi \nonumber\\
    &=D^2\psi -\frac{1}{2}e_i\cdot \nabla_i((\nabla_j f)e_j\cdot \psi)-\frac{1}{2}(\nabla_j f)e_j\cdot e_i\cdot \nabla_i\psi -\frac{1}{4}|\nabla f|^2\psi \nonumber\\
    &=D^2\psi 
    -\frac{1}{2} (\nabla_i\nabla_j f)e_i\cdot e_j\cdot \psi
    -\frac{1}{2}(\nabla_j f)(e_i\cdot e_j+e_j\cdot e_i)\cdot \nabla_i\psi 
    -\frac{1}{4}|\nabla f|^2\psi \nonumber\\
    &=-\Delta \psi+\frac{1}{4}\Scal \psi 
    +\frac{1}{2} (\Delta f) \psi
    +\langle \nabla f,\nabla\psi \rangle
    -\frac{1}{4}|\nabla f|^2\psi \nonumber\\
    &=-\Delta_f\psi+\frac{1}{4}(\Scal+2\Delta f-|\nabla f|^2)\psi \nonumber\\
    &=-\Delta_f\psi+\frac{1}{4}\Scal_f\psi. \nonumber
\end{align}
\end{proof}

\begin{remark}
    The weighted Lichnerowicz formula also follows from the Lichnerowicz formula for spin-c Dirac operators \cite[\S 3.3]{F2},
    \begin{equation}
        D^2_A=-\Delta_A+\frac{1}{4}\Scal +\frac{1}{2}dA,
    \end{equation}
    by choosing the spin-c connection $\nabla^A$ for which $A=-\frac{1}{2}df$.
    Indeed, with this connection, 
    \begin{equation}
        \Delta_{A}
        =(\nabla^A)^*\nabla^A
        =\Delta_f-\frac{1}{4}(2\Delta f-|\nabla f|^2)
    \end{equation}
    and $dA=-\frac{1}{2}d^2f=0$, from which the weighted Lichnerowicz formula (\ref{eqn: weighted Lichnerowicz}) follows immediately. In this sense, the weighted Dirac operator can also be thought of as the twisted Dirac operator $D_A$.
\end{remark}

\begin{proposition}
    [Weighted Ricci identity]\label{prop: weighted Ricci identity}
    The weighted Ricci curvature $\Ric_f=\Ric+\Hess_f$ is proportional to the commutator of $D_f$ and $\nabla$: for any vector field $X$ and spinor $\psi$, 
    \begin{equation}
        [D_f,\nabla_X]\psi=\frac{1}{2}\Ric_f(X)\cdot\psi.
    \end{equation}
\end{proposition}

\begin{proof}
Recall the unweighted Ricci identity, $[D,\nabla_X]=\frac{1}{2}\Ric(X)\cdot$. (For a proof, see for example \cite[Rem.\ 2.50]{BHMMM}). Using this identity and the fact that Clifford multiplication is parallel with respect to the weighted spin connection (\ref{eqn: Cliff mult is weighted parallel}), it follows that, for any spinor $\psi$,
\begin{align}
    D_f\nabla_X\psi-\nabla_XD_f\psi
    &=D\nabla_X\psi-\frac{1}{2}(\nabla f)\cdot \nabla_X \psi
        -\nabla_XD\psi
        +\frac{1}{2}\nabla_X((\nabla f)\cdot\psi) \\
    &=[D,\nabla_X]\psi 
        +\frac{1}{2}(\nabla_X\nabla f)\cdot\psi \nonumber \\
    &= \frac{1}{2}\Ric(X)\cdot \psi 
        +\frac{1}{2}\Hess_f(X)\cdot \psi. \nonumber
\end{align}
\end{proof}

In what follows, denote the space of weighted $L^2$-spinors by $L_f^2=L^2(\Sigma M,e^{-f}dV)$ and let $L^2$ be the space of unweighted $L^2$-spinors. Define the linear operator
\begin{equation}
    U_f:L^2\to L^2_f, \qquad \qquad \psi\mapsto e^{f/2}\psi.
\end{equation} 
This operator is an isomorphism of Hilbert spaces with inverse given by $U_f^{-1}=U_{-f}$; it preserves norms since
\begin{equation}
    \|U_f\psi\|_{L^2_f}
    =\int_M |e^{f/2}\psi|^2\,e^{-f}dV
    =\|\psi\|_{L^2}.
\end{equation}
In particular, $U_f$ is a unitary operator. Recall that two operators $A,B$ acting on Hilbert spaces with domains of definition $\mathcal{D}_A$ and $\mathcal{D}_B$ are unitarily equivalent if there exists a unitary operator $U$ such that $U\mathcal{D}_A=\mathcal{D}_B$ and $UAU^{-1}x=Bx$ for all $x\in \mathcal{D}_B$.

\begin{proposition}
[Unitary equivalence]\label{prop: unitary equivalence}
    The Dirac operator $D$ and the weighted Dirac operator $D_f$ are unitarily equivalent and hence isospectral;
    on $C^1$-spinors, these operators are related by
    \begin{equation}\label{eqn: unitary equivalence of Dirac and weighted Dirac}
        U_fDU_f^{-1}=D_f.
    \end{equation}
    In particular, $D\psi=0$ if and only if $D_f(e^{f/2}\psi)=0$.
\end{proposition}

\begin{proof}
For any $C^1$-spinor $\psi$, 
    \begin{align}
        U_fDU_f^{-1}\psi
        =e^{f/2}D(e^{-f/2}\psi) 
        =e^{f/2}\left(e^{-f/2}D\psi+(\nabla e^{-f/2})\cdot \psi\right) 
        =D\psi -\frac{1}{2}(\nabla f)\cdot \psi 
        =D_f\psi. 
    \end{align}
This proves (\ref{eqn: unitary equivalence of Dirac and weighted Dirac}), and it follows immediately from this equation and the fact that $U_f$ is an isomorphism, that $D\psi=\lambda \psi$ if and only if $D_f(U_f\psi)=\lambda U_f\psi$. 
In particular, $U_f$ is an isomorphism between the eigenspaces $E_{\lambda}(D)$ and $E_{\lambda}(D_f)$, for any $\lambda\in \R$. Hence, (when defined) the multiplicities of the eigenvalues coincide.
\end{proof}

The following eigenvalue inequality is a generalization of Friedrich's inequality \cite{F1} and the proof below generalizes his proof. See \cite[\S 5.1]{F2} for an insightful exposition of the classical proof, whose outline will be followed below. The weighted Friedrich inequality proved below is sharp. Indeed, on the round sphere with constant scalar curvature $\Scal$ and with $f$ a constant function, equality is obtained. 

\begin{theorem}\label{thm: weighted friedrich}
    Suppose that $(M^n,g)$ is closed, let $f\in C^{\infty}(M)$, and let $\lambda$ be an eigenvalue of the Dirac operator $D$. Then
    \begin{equation}
        \lambda^2\geq \frac{n}{4(n-1)}\min \Scal_f, \label{eqn: weighted friedrich}
    \end{equation}
    with equality if and only if $f$ is constant and $(M^n,g)$ admits a Killing spinor, in which case $(M^n,g)$ is Einstein.
\end{theorem}

\begin{proof}
Let $\psi$ be an eigenspinor of the Dirac operator with $D\psi=\lambda \psi$.
Define the connection
\begin{equation}
    \nabla^{f,\lambda}_X=\nabla_X+ \frac{1}{2}(\nabla_Xf)+\frac{1}{2n}X\cdot (\nabla f) \cdot +\frac{\lambda}{n}X\cdot.
\end{equation}
A calculation employing a local orthonormal frame shows that the assumption $D\psi=\lambda\psi$ implies
\begin{equation}
    |\nabla^{f,\lambda}\psi|^2=|\nabla \psi|^2-\frac{\lambda^2}{n}|\psi|^2+\frac{1}{4}\left(1-\frac{1}{n}\right)|\nabla f|^2|\psi|^2+\frac{1}{2}\langle \nabla f,\nabla |\psi|^2\rangle.
\end{equation}
Integrating the above equation over $M$ and integrating the last term by parts implies
\begin{equation}
    \int_M|\nabla^{f,\lambda}\psi|^2\,dV
    =\int_M\left(|\nabla \psi|^2-\frac{\lambda^2}{n}|\psi|^2+\frac{1}{4}\left(1-\frac{1}{n}\right)|\nabla f|^2|\psi|^2-\frac{1}{2}(\Delta f) |\psi|^2\rangle\right)dV.
\end{equation}
The standard (unweighted) Lichnerowicz formula, the self-adjointness of $D$ on $L^2$, and the definition of the weighted scalar curvature $\Scal_f$ then imply
\begin{align}
    \int_M|\nabla^{f,\lambda}\psi|^2\,dV
    &=\int_M\left(|D \psi|^2-\frac{1}{4}\Scal|\psi|^2-\frac{\lambda^2}{n}|\psi|^2+\frac{1}{4}\left(1-\frac{1}{n}\right)|\nabla f|^2|\psi|^2-\frac{1}{2}(\Delta f) |\psi|^2\rangle\right)dV
    \\
    &=\int_M\left(\left(\frac{n-1}{n}\right)\lambda^2|\psi|^2-\frac{1}{4}\Scal_f|\psi|^2-\frac{1}{4n}|\nabla f|^2|\psi|^2\right)dV, \nonumber
\end{align}
which, after rearranging, implies
\begin{align}
    \lambda^2\left(\frac{n-1}{n}\right)\|\psi\|_{L^2}^2
    &=\|\nabla^{f,\lambda}\psi\|_{L^2}^2
        +\frac{1}{4}\int_M\left(\Scal_f+\frac{1}{n}|\nabla f|^2\right)|\psi|^2\,dV \\
    &\geq \frac{1}{4}\min_M \Scal_f \|\psi\|_{L^2}^2. \nonumber
\end{align}
This was to be shown.

If equality occurs in the previous inequality, then $\Scal_f$ is constant, $\nabla^{f,\lambda}\psi=0$ and $\nabla f=0$. In particular, $f$ is constant, so $0=\nabla^{f,\lambda}\psi=\nabla^{0,\lambda}\psi$. This is equivalent to the condition that
\begin{equation}
    \nabla_X\psi=-\frac{\lambda}{n}X\cdot \psi,
\end{equation}
for all vector fields $X$. Hence $\psi$ is a Killing spinor.
Finally, a manifold admitting a Killing spinor must be Einstein; see for example \cite[\S 5.2]{F2}. The converse is immediate.
\end{proof}

Whenever the scalar curvature is not constant, Theorem \ref{thm: weighted friedrich} implies a strict improvement of Friedrich's inequality.
This is because the weight $f$ can always be chosen to make $\Scal_f$ constant, while if $\Scal$ is not constant, then it follows that $\Scal_f>\Scal_{\min}$.
To show this, recall that Perelman's entropy $\lambda_{\mathrm{P}}$ is defined as the first eigenvalue of the operator $-4\Delta +\Scal$, or equivalently, as the minimum of the weighted Hilbert-Einstein functional \cite{P}:
\begin{equation}
    \lambda_{\mathrm{P}}
    =\inf_u\frac{\int_M\left(4|\nabla u|^2+\Scal u^2\right)dV}{\int_M u^2\,dV}
    =\inf_f\frac{\int_M\Scal_fe^{-f} \,dV}{\int_Me^{-f}\,dV}.
\end{equation}
If $f$ is the minimizer of $\lambda_{\mathrm{P}}$, the \emph{weighted} scalar curvature is constant, with $\Scal_f=\lambda_{\mathrm{P}}$.
On the other hand, if the scalar curvature is not constant, then  $\Scal_f=\lambda_{\mathrm{P}}>\Scal_{\min}$, and thus the weighted Friedrich inequality (\ref{eqn: weighted friedrich}) implies a strict improvement of Friedrich's inequality. 

\begin{corollary}\label{thm: weighted friedrich at lambda}
    Any eigenvalue $\lambda$ of the Dirac operator $D$ on a closed manifold $(M^n,g)$ satisfies
    \begin{equation}
        \lambda^2\geq \frac{n}{4(n-1)}\lambda_{\mathrm{P}}(g),
    \end{equation}
    with equality if and only if $(M^n,g)$ admits a Killing spinor, in which case $(M^n,g)$ is Einstein.
\end{corollary}

The bound (\ref{thm: weighted friedrich at lambda}) gives another proof of the \emph{stability} of hyperk\"ahler metrics on the $K3$ surface along Ricci flow. Indeed, all metrics on $K3$ satisfy the above inequality with $\lambda = 0$ since $\hat{A}(K3)\neq 0$. Consequently, Corollary \ref{thm: weighted friedrich at lambda} implies that $\lambda_P(g)\leqslant 0$ for \emph{all} metrics $g$ on $K3$, with equality exactly on hyperk\"ahler metrics. These metrics are consequently stable by \cite{Ha}.

\begin{remark}
    Hijazi \cite[Eqn.\ (5.1)]{Hi} proved an inequality closely related to that of Theorem \ref{thm: weighted friedrich}. Hijazi's proof employs the Dirac operator of a conformally related metric, whereas the proof of Theorem \ref{thm: weighted friedrich} keeps the metric fixed and uses the weighted Lichnerowicz formula (\ref{eqn: weighted Lichnerowicz}). 
    Hijazi's inequality implies that any eigenvalue $\lambda$ of the Dirac operator satisfies $\lambda^2\geq \frac{n}{4(n-1)}\mu_1(g)$, where $\mu_1(g)$ is the smallest eigenvalue of the conformal Laplace operator $-4\frac{n-1}{n-2}\Delta +\Scal$. Since $\lambda_{\mathrm{P}}(g)$ is the first eigenvalue of the operator $-4\Delta+\Scal$, it follows that
    \begin{equation}
        \mu_1(g)\geq \lambda_{\mathrm{P}}(g).
    \end{equation}
    In this sense, Hijazi's inequality \cite[Eqn.\ (5.1)]{Hi} is sharper than the inequality of Theorem \ref{thm: weighted friedrich}. On the other hand, the inequality in Corollary \ref{thm: weighted friedrich at lambda} improves along Ricci flow.
\end{remark}
\section{Weighted asymptotically Euclidean manifolds}\label{sec: weighted AE manifolds}

A smooth orientable Riemannian manifold $(M^n,g)$ is called asymptotically Euclidean (AE) of order $\tau$ if there exists a compact subset $K\subset M$ and a diffeomorphism $\Phi:M\setminus K\to \R^n\setminus B_{\rho}(0)$, for some $\rho>0$, with respect to which
\begin{equation}
    g_{ij}=\delta_{ij}+O(r^{-\tau}), \qquad \partial^kg_{ij}=O(r^{-\tau-k}),
\end{equation}
for any partial derivative of order $k$ as $r\to \infty$, where $r=|\Phi|$ is the Euclidean distance function. 
The set $M\setminus K$ is called the end of $M^n$. (The results of this section extend in a straightforward manner to AE manifolds with multiple ends, though they are not pursued here.)

The ADM mass \cite{ADM} of $(M^n,g)$ is defined 
by
\begin{equation}\label{eqn: definition of mass}
    \mass(g)
    =\lim_{\rho\to \infty}\int_{S_{\rho}}(\partial_ig_{ij}-\partial_jg_{ii})\, \partial_j \intprod dV_{g},
\end{equation}
where $S_{\rho}=r^{-1}(\rho)$ is a coordinate sphere of radius $\rho$.\footnote{
The ADM mass as defined in \cite{Ba86} equals $(2(n-1)\omega_{n-1})^{-1} \mass(g)$, where $\omega_{n-1}$ is the area of the unit sphere in $\R^n$.
}
Although the definition of mass involves a choice of AE coordinates, if $\tau>(n-2)/2$ and the scalar curvature is integrable, then the mass is finite and independent of the choice of AE coordinates \cite{Ba86, C}. 
If $n\leq 7$ or $(M^n,g)$ admits a spin structure, then the assumptions $\Scal\geq 0$, $\Scal\in L^1(M,g)$, and $\tau>\frac{n-2}{2}$, imply that $\mass(g)$ is nonnegative and is zero if and only if $(M^n,g)$ is isometric to $(\R^n,g_{\euc})$, by the positive mass theorem \cite{SY,W}.

The AE structure defines a trivialization of the spin bundle at infinity. Indeed, choose an asymptotic coordinate system $\Phi^{-1}:\R^n\setminus B_R(0)\to M\setminus K$. The pullback bundle $(\Phi^{-1})^*\Sigma M$ differs from the trivial spin bundle $\R^n\times \Sigma$ by an element of $H^1(\R^n\setminus B_R(0);\Z)=0$. 
Hence the spin structure is trivial over the end of $M$ and the bundle $(\Phi^{-1})^*\Sigma M$ extends trivially over all of $\R^n$. 
This trivialization of the spin bundle allows for the definition of ``constant spinors'' on the end of $M$: a spinor $\psi$ defined on the end $M$ is called {\it constant} (with respect to the asymptotic coordinates $\Phi$) if $\psi=(\Phi^{-1})^*\psi_0$, for some constant spinor $\psi_0$ on $\R^n$.

Witten argued that for any such constant spinor $\psi_0$ on $M\setminus K$ with $|\psi_0|\to 1$ at infinity, there exists a harmonic spinor $\psi$ on $M$ which is asymptotic to $\psi_0$, in the sense that $|\psi-\psi_0|=O(r^{-\tau})$ and $|\nabla \psi|=O(r^{-\tau-1})$. Such a spinor $\psi$ is called a {\it Witten spinor}. Moreover, the ADM mass of $(M^n,g)$ is given by
\begin{equation}
    \mass(g)=4\int_M\left(|\nabla \psi|^2+\frac{1}{4}\Scal |\psi|^2\right) dV_g,
\end{equation}
which is called {\it Witten's formula} for the mass. A rigorous proof of the existence of Witten spinors is given by Parker-Taubes \cite{PT} and Lee-Parker \cite{LP}; their proofs are generalized below and in Appendix \ref{sec: appendix}.

\subsection{Weighted mass.}\label{subsec: weighted mass}

The {\it weighted ADM mass} of a weighted AE manifold $(M^n,g,f)$ is defined by
\begin{equation}\label{eqn: weighted mass defn}
    \mass_f(g) := \mass(g) + 2\lim_{\rho\to \infty}\int_{S_{\rho}}\langle \nabla f,\nu\rangle \,e^{-f}dA.
\end{equation}
This definition is motivated by the weighted Witten formula (\ref{eqn: weighted Witten formula}) below, and manifestly extends to non-spin manifolds. 
Like ADM mass, the weighted mass is independent of the choice of asymptotic coordinates if $\tau>\frac{n-2}{2}$ and $\Scal\in L^1(M)$: indeed, the ADM mass is coordinate independent under said assumptions \cite{Ba86, C}, and by the divergence theorem, the second term in (\ref{eqn: weighted mass defn}) equals $2\int_M(\Delta_ff) \, e^{-f}dV$, which is manifestly coordinate independent.

The appropriate analytic tools for studying AE manifolds are the \emph{weighted H\"older spaces} $C^{k,\alpha}_{\beta}(M)$, whose precise definitions are stated in Appendix \ref{sec: appendix}.
These spaces share many of the global elliptic regularity results which hold for the usual H\"older spaces on compact manifolds.
The index $\beta$ is important because it denotes the {\it order of growth}: functions in $C^{k,\alpha}_{\beta}(M)$ grow at most like $r^{\beta}$.
In particular, if the metric $g$ is AE of order $\tau$ on $M=\R^n$, then in the AE coordinate system, $g-\delta$ lies in $C^{k,\alpha}_{-\tau}(M)$ for all $k\in \N$ and the scalar curvature of $g$ lies in $C^{k,\alpha}_{-\tau-2}(M)$ for all $k\in \N$.

In what follows, let $D_f$ be the weighted Dirac operator associated with the weighted spin connection (\ref{eqn: weighted spin connection}) defined by $f$, which satisfies the weighted Lichnerowicz formula (\ref{eqn: weighted Lichnerowicz}). 

\begin{theorem}[Weighted Witten formula]\label{thm: weighted witten}
Let $(M^n,g,f)$ be a weighted, spin, AE manifold of order $\tau$. Suppose that $f\in C^{2,\alpha}_{-\tau}(M)$, that  
\begin{align}\label{eqn: positive weighted mass assumptions}
    \Scal_f \geqslant  0, \qquad &\Scal_f \in L^1(M,g),  \qquad \frac{n-2}{2}<\tau<n-2,
\end{align}
and that $\psi_0$ is a spinor on $(M^n,g)$ which is constant at infinity, with $|\psi_0|\to 1$. Then there exists a $D_f$-harmonic spinor $\psi$ which is asymptotic to $\psi_0$ in the sense that $\psi-\psi_0\in C^{2,\alpha}_{-\tau}(M)$ and
\begin{equation}\label{eqn: weighted Witten formula}
    \mass_f(g)
        =4\int_M\left(|\nabla \psi|^2+\frac{1}{4}\Scal_f|\psi|^2\right)e^{-f} dV_g.
\end{equation}
\end{theorem}
\begin{proof}
Here the proof is given under the additional natural assumptions that $\Scal\geq 0$, $\Scal\in L^1(M,g)$ and $|\nabla f|=O(r^{-(n-1)})$. The additional assumptions $\Scal\geq 0$, $\Scal\in L^1(M,g)$ ensure the existence of an (unweighted) Witten spinor $\psi$ \cite{LP}. Further, the assumption $|\nabla f|=O(r^{-(n-1)})$ is satisfied if $\Scal_f=0$; see \cite[Prop.\ (2.2)]{DO20}. In Appendix \ref{subsec: appendix Witten}, a proof of the general case is given.

By (\ref{eqn: unitary equivalence of Dirac and weighted Dirac}), if $D\psi=0$, then the spinor $\psi_f=e^{f/2}\psi$ is $D_f$-harmonic. Since 
\begin{align}
    \nabla\psi
    &=\nabla(e^{-f/2}\psi_f)
    =e^{-f/2}\left(\nabla \psi_f-\frac{1}{2}df\otimes \psi_f\right), \\
    \nabla \psi_f
    &=\nabla (e^{f/2}\psi)=e^{f/2}\nabla \psi+\frac{1}{2}df\otimes \psi_f,
\end{align}
it follows that
\begin{align}
    |\nabla\psi|^2
    &=e^{-f}\left(|\nabla \psi_f|^2+\frac{1}{4}|\nabla f|^2|\psi_f|^2-\Rea\langle \nabla \psi_f,df\otimes \psi_f\rangle \right) \\
    &=e^{-f}\left(|\nabla \psi_f|^2+\frac{1}{4}|\nabla f|^2|\psi_f|^2-\frac{1}{2}|\nabla f|^2|\psi_f|^2-e^{f}\Rea\langle \nabla \psi,df\otimes \psi\rangle\right) \nonumber\\
    &=e^{-f}\left(|\nabla \psi_f|^2-\frac{1}{4}|\nabla f|^2|\psi_f|^2\right)-\Rea\langle \nabla_{\nabla f} \psi, \psi\rangle. \nonumber
\end{align}
By the definition of the weighted scalar curvature (\ref{eqn: defn of weighted scalar curvature}) and Witten's formula for the mass,
\begin{align}
    \frac{1}{4}\mass(g)
    &=\int_M\left(|\nabla \psi|^2+\frac{1}{4}\Scal |\psi|^2\right) dV_g \\
    &=\int_M\left(|\nabla \psi_f|^2+\frac{1}{4}(\Scal-|\nabla f|^2) |\psi_f|^2\right) e^{-f}dV_g 
        -\Rea\int_M\langle \nabla_{\nabla f} \psi, \psi\rangle \,dV_g \nonumber\\
    &=\int_M\left(|\nabla \psi_f|^2+\frac{1}{4}\Scal_f|\psi_f|^2-\frac{1}{2}(\Delta f)|\psi_f|^2\right) e^{-f}dV_g 
        -\Rea\int_M\langle \nabla_{\nabla f} \psi, \psi\rangle \,dV_g \nonumber\\
    &=\int_M\left(|\nabla \psi_f|^2+\frac{1}{4}\Scal_f|\psi_f|^2\right) e^{-f}dV_g 
        -\int_M \left(\frac{1}{2}(\Delta f)|\psi|^2+\Rea\langle\nabla_{\nabla f} \psi, \psi\rangle \right) dV_g. \nonumber
\end{align}
Integrating the last term by parts and using the fact that $|\psi_f|\to 1$ at infinity gives
\begin{align}
    \frac{1}{4}\mass(g)
    &=\int_M\left(|\nabla \psi_f|^2+\frac{1}{4}\Scal_f|\psi_f|^2\right) e^{-f}dV_g 
        -\lim_{\rho\to \infty}\frac{1}{2}\int_{S_{\rho}}\langle \nabla f,\nu\rangle |\psi_f|^2\,e^{-f}dA \\
        &=\int_M\left(|\nabla \psi_f|^2+\frac{1}{4}\Scal_f|\psi_f|^2\right) e^{-f}dV_g 
        -\lim_{\rho\to \infty}\frac{1}{2}\int_{S_{\rho}}\langle \nabla f,\nu\rangle\,e^{-f}dA. \nonumber
\end{align}
By the assumption $f\to 0$ at infinity and $|\nabla f|=O(r^{-(n-1)})$, the latter limit exists and is finite, since the area of $S_{\rho}$ is of order $\rho^{n-1}$. 
\end{proof}

The following theorem generalizes Schoen-Yau \cite{SY} and Witten's \cite{W} positive mass theorem to the weighted (spin) setting. 

\begin{theorem}
[Positive weighted mass theorem]\label{thm: positive weighted mass}
Let $(M^n,g,f)$ be a weighted, spin, AE manifold satisfying the assumptions of Theorem \ref{thm: weighted witten}.
Then $\mass_f(g)\geqslant  0$, with equality if and only if $(M^n,g)$ is isometric to $(\R^n,g_{\euc})$ and $\int_{\R^n}(\Delta_f f)\,e^{-f}dV=0$.
\end{theorem}

\begin{proof}
    Theorem \ref{thm: weighted witten} provides the existence of a weighted Witten spinor $\psi$ satisfying the weighted Witten formula (\ref{eqn: weighted Witten formula}).
    This shows that $\mass_f(g)\geq 0$ if $\Scal_f\geq 0$. The proof of the equality statement resembles Witten's proof of the equality statement for the positive mass theorem: equality implies that $\nabla \psi=0$, and since there exist $\mathrm{rank}( \Sigma M)$ possible linearly independent constant spinors at infinity $\psi_0$ to which $\psi$ is asymptotic, $\Sigma M$ admits a basis of parallel spinors. Since the map $\Sigma M\to TM$ sending a spinor $\phi$ to the vector field $V_{\phi}$ defined by
    \begin{equation}
        \langle V_{\phi},X\rangle =\Ima\langle \phi,X\cdot \phi\rangle \qquad \text{for all } X\in \Gamma(TM),
    \end{equation}
    is surjective, and since $V_{\phi}$ is a parallel vector field if $\phi$ is a parallel spinor, $TM$ admits a basis of parallel vector fields. Thus $(M^n,g)$ is flat.
    
    Finally, since $\mass(g_{\euc})=0$, integration by parts and $\mass_f(g_{\euc})=0$ imply that 
    \begin{align}
        0
        &=\mass_f(g_{\euc})
        =\lim_{\rho\to \infty}2\int_{S_{\rho}}\langle \nabla f,\nu\rangle \,e^{-f}dA 
        =-2\int_{\R^n}(\Delta_f f)\,e^{-f}dV.
    \end{align}
\end{proof}

\subsection{Weighted mass and Ricci flow.}

Given an asymptotically Euclidean manifold $(M^n,g)$, define the \emph{renormalized Perelman entropy} as
\begin{equation}\label{eqn: defn of lambdaALE}
    \lambdaALE(g)=\inf_{u-1\in C^{\infty}_c(M)}\int_M\left(4|\nabla u|^2+\Scal u^2\right)dV -\mass(g).
\end{equation}
Note that $\lambdaALE(g)$ can equivalently be defined as the infimum of $\int_M\Scal_f\,e^{-f}dV-\mass_f(g)$, over all $f\in C^{\infty}_c(M)$.
If $(M^n,g)$ admits a Witten spinor $\psi$, then testing the right-hand-side of the above equation with $u=|\psi|$ gives that $\lambdaALE(g)\leq 0$, by Kato's inequality, $|\nabla |\psi||\leq |\nabla \psi|$. As mentioned in the Introduction, Ricci flow is the gradient flow of $\lambdaALE$ on AE manifolds and $\lambdaALE$ has various advantages over the ADM mass in the context of Ricci flow; see the Introduction and also \cite{DO20}.

\begin{theorem}\label{thm: weighted mass = lambda}
Let $(M^n,g)$ be an asymptotically Euclidean manifold of order $\tau>\frac{n-2}{2}$ and with nonnegative scalar curvature. 
Then there exists a solution $f\in C^{2,\alpha}_{-\tau}(M)$ of the elliptic equation $\Scal_{f}=0$, and the $f$-weighted mass satisfies
\begin{equation}
    \mass_f(g)=-\lambdaALE(g).
\end{equation}
\end{theorem}

\begin{proof}
By \cite[(2.3)]{DO20}, there exists a strictly positive minimizer $w=e^{-f/2}$ of (\ref{eqn: defn of lambdaALE}) with $w-1\in C^{2,\alpha}_{-\tau}(M)$ satisfying $-4\Delta w +\Scal w=0$. 
Since $w\to 1$ at infinity, integration by parts implies
\begin{align}
    \inf_{u-1\in C^{\infty}_c(M)}\int_M\left(4|\nabla u|^2+\Scal u^2\right)dV
    &=\int_M\left(4|\nabla w|^2+\Scal w^2\right)dV  \\      
    &=\lim_{\rho\to \infty}\int_{S_{\rho}}4\langle \nabla w,\nu\rangle w\,dA \nonumber \\      
    &=-2\lim_{\rho\to\infty}\int_{S_{\rho}}\langle \nabla f,\nu\rangle \,e^{-f}dA. \nonumber 
\end{align}
The result now follows immediately from the definition (\ref{eqn: weighted mass defn}) of $\mass_f(g)$ and that of $\lambdaALE$, (\ref{eqn: defn of lambdaALE}).

Note that \cite[Eqn.\ (2.3)]{DO20} is stated for ALE manifolds in the neighborhood of a Ricci-flat ALE manifold, to ensure the existence and uniqueness of $f$ by the positivity of $-4\Delta+\Scal$ thanks to a Hardy inequality; see \cite[Prop.\ 1.12]{DO20}. However, the same proof holds under the above assumptions on $(M^n,g)$ since the scalar curvature is nonnegative and the operator $-4\Delta+\Scal$ is therefore positive; see the proof of \cite[Thm.\ 2.6]{Ha} for a similar argument.
\end{proof}

It has been proven in \cite[Thm.\ 2.2]{Li18} that the AE conditions are preserved along Ricci flow (with the same coordinate system) as long as the flow is nonsingular. An \emph{asymptotically Euclidean Ricci flow} is defined to be any Ricci flow starting at an AE manifold.
\begin{corollary}
[Monotonicity of weighted mass]\label{cor: monotonicity of weighted mass}
Let $(M^n,g(t))_{t\in I}$ be an asymptotically Euclidean Ricci flow with nonnegative scalar curvature.
Let $f:M\times I\to \R$ be the time-dependent family of functions solving $\Scal_{f}=0$ and $f\to 0$ at infinity, at each time $t\in I$. Then 
\begin{equation}\label{eqn: time derivative of weighted mass}
    \frac{d}{dt}\mass_{f}(g)
    =-2\int_M|\Ric+\Hess_{f}|^2e^{-f} dV \leq 0.
\end{equation}
In particular, $\mass_{f}(g)$ is monotone decreasing along the Ricci flow, and is constant only if $(M^n,g(t))$ is Ricci-flat.
\end{corollary}

\begin{proof}
Since $\mass_f(g)=-\lambda_{ALE}(g)$, equation (\ref{eqn: time derivative of weighted mass}) follows from the formula for the first variation of $\lambdaALE$, which can be found in \cite[Prop.\ 2.3 and 3.13]{DO20}.
Once again, the assumptions of closeness to a Ricci-flat ALE metric of Deruelle-Ozuch can be replaced by the nonnegativity of scalar curvature. 
Their closeness assumption is again only used to ensure the existence of $f$. Note that in contrast with Perelman's monotonicity for closed manifolds, which is proved by letting $f$ evolve {\it parabolically} backwards in time, the monotonicity formula (\ref{eqn: time derivative of weighted mass}) uses the fact that $f$ solves the {\it elliptic} equation $\Scal_f=0$ at each time.

To prove the equality statement, note that formula \eqref{eqn: time derivative of weighted mass} implies that $\mass_f(g)$ is constant if and only if $(M^n,g,f)$ is a steady Ricci soliton. The proof is completed by using \cite[Prop.\ 2.6]{DK}: any ALE steady soliton with $\nabla f\to 0$ at infinity is Ricci flat.
\end{proof}

\appendix
\section{Appendix}
\label{sec: appendix}
This appendix provides a proof of the general case of Theorem \ref{thm: weighted witten} on the existence of a weighted Witten spinor satisfying the weighted Witten formula. In Section \ref{subsec: weighted mass}, a simple and illustrative proof was given under natural, albeit more restrictive assumptions.

Let $(M^n,g)$ be an asymptotically Euclidean (AE), Riemannian spin manifold of order $\tau$. The asymptotic coordinates define a positive function $r$ on $M$, which equals the Euclidean distance to the origin on $M\setminus K$, and which can be extended to a smooth function which is bounded below by 1 on all of $M$. 

Using $r$, the weighted $C^k$ space $C^k_{\beta}(M)$ is defined for $\beta\in \R$ as the set of $C^k$ functions $u$ on $M$ for which the norm
\begin{equation}
    \|u\|_{C^k_{\beta}}
        =\sum_{i=0}^k\sup_Mr^{-\beta +i}|\nabla^iu|
\end{equation}
is finite. The weighted H\"older space $C^{k,\alpha}_{\beta}(M)$ is defined for $\alpha \in (0,1)$ as the set of $u\in C^k_{\beta}(M)$ for which the norm
\begin{equation}
    \|u\|_{C^{k,\alpha}_{\beta}}
        =\|u\|_{C^k_{\beta}}
            +\sup_{x,y}\;(\min\{r(x),r(y)\})^{-\beta+k+\alpha}\frac{|\nabla^ku(x)-\nabla^ku(y)|}{d(x,y)^{\alpha}}
\end{equation}
is finite.\footnote{The meaning of ``weighted'' in ``weighted H\"older spaces'' is distinct from its meaning in ``weighted manifolds.''} These definitions of weighted H\"older spaces coincide with those of \cite[\S 9]{LP}. In particular, the index $\beta$ denotes the {\it order of growth}: functions in $C^{k,\alpha}_{\beta}(M)$ grow at most like $r^{\beta}$. Note that the definitions of the weighted function spaces depend on the ``distance function'' $r$, and thereby on the choice of asymptotic coordinates. However, it is easy to see that $r$ is uniformly equivalent to the geodesic distance from an arbitrary fixed point in $M$ as $r\to \infty$, hence all choices of $r$ define equivalent norms. For the remainder of this appendix, fix $\alpha\in (0,1)$.

\subsection{Existence of weighted Witten spinors.}
\label{subsec: appendix Witten}

Let $f\in C^{\infty}(M)$ and let $D_f$ be the weighted Dirac operator associated with the weighted spin connection (\ref{eqn: weighted spin connection}) defined by $f$, which satisfies the weighted Lichnerowicz formula (\ref{eqn: weighted Lichnerowicz}).

\begin{lemma}\label{lem: D^2_f is an iso}
On a weighted, AE, spin manifold $(M^n,g,f)$ satisfying the hypotheses of Theorem \ref{thm: weighted witten}, the operator
\begin{equation}
    D_f^2:C^{2,\alpha}_{-\tau}(M)\to C^{0,\alpha}_{-\tau-2}(M)
\end{equation}
is an isomorphism.
\end{lemma}

\begin{proof}
To show injectivity, suppose $D_f^2\xi=0$ for some $\xi \in C^{2,\alpha}_{-\tau}(M)$. Then $\xi=O(r^{-\tau})$ and $\nabla \xi=O(r^{-\tau-1})$. Applying the weighted Lichnerowicz formula and integration by parts (the boundary term vanishes because $\tau > (n-2)/2$), it follows that
\begin{align}
    0
    &=\int_M\langle D^2_f\xi,\xi\rangle \,e^{-f} dV_g 
    =\int_M\left(-\langle \dL_f\xi,\xi\rangle +\frac{1}{4}\Scal_f |\xi|^2 \right)e^{-f} dV_g 
    =\int_M \left( |\nabla \xi|^2+\frac{1}{4}\Scal_f |\xi|^2\right)e^{-f} dV_g.
\end{align}
Since $\Scal_f\geq 0$, this shows that $\nabla \xi=0$, so $\nabla |\xi|^2=0$. Thus $|\xi|$ is a constant, which is zero since $\xi$ vanishes at infinity. Thus $D_f^2$ is injective.

The weighted Lichnerowicz formula implies that $D_f^2=-\Delta +\nabla_{\nabla f}+\frac{1}{4}\Scal_f$. Since $f\in C^{2,\alpha}_{-\tau}$ and $g$ is smooth and AE of order $\tau$, it follows that $\nabla f\in C^{1,\alpha}_{-\tau-1}$ and $\Scal_f\in C^{0,\alpha}_{-\tau-2}$. Since $\frac{n-2}{2}<\tau<n-2$, it follows from \cite{CSCB} that $D_f^2$ is an isomorphism if it is injective. With injectivity proven above, the proof is complete; see \cite[Thm.\ 9.2d]{LP} for the proof for the unweighted Dirac operator.
\end{proof}

As explained in Section \ref{sec: weighted AE manifolds}, the asymptotically Euclidean structure defines a trivialization of the spin bundle at infinity, allowing for the notion of a spinor which is ``constant'' in the asymptotic coordinate system. In what follows, for $\rho>0$, let $S_{\rho}=r^{-1}(\rho)$ be the $\rho$-level set of $r$, that is, a coordinate sphere of radius $\rho$.

\begin{proof}
[Proof of Theorem \ref{thm: weighted witten}]
With respect to the trivialization of the spin bundle at infinity, the weighted Dirac operator may be written as 
\begin{equation}\label{eqn: asymptotic expansion of Dirac operator}
    D_f
    =e^i\cdot \partial_i-\frac{1}{2}(\nabla f)\cdot  -\frac{1}{8}(\partial_k g_{ij})e^i\cdot [e^j\cdot ,e^k\cdot] +O(r^{-2\tau-1}).
\end{equation}

Choose a spinor $\psi_0$ which is constant at infinity and with $|\psi_0|\to 1$ at infinity, and extend it to a smooth spinor on $M$. 
It follows from the above equation and the assumption $f\in C^{2,\alpha}_{-\tau}(M)$ that $D_f^2\psi_0\in C^{0,\alpha}_{-\tau-2}(M)$.
By Lemma (\ref{lem: D^2_f is an iso}), there exists $\xi\in C^{2,\alpha}_{-\tau}(M)$ with $D^2_f\xi=D^2_f\psi_0$. The spinor $\psi=\psi_0-\xi$ then satisfies $D_f^2\psi=0$ and $\phi:=D_f\psi=D_f\psi_0-D_f\xi$ satisfies $D_f\phi=0$ and lies in $C^{1,\alpha}_{-\tau-1}(M)$, so integrating by parts as in the proof of the Lemma above shows that $\phi=0$. Thus $\psi$ is a weighted harmonic spinor which is asymptotic to $\psi_0$.

Let $X$ be the vector field on $M\setminus K$ defined by
\begin{equation}
    X=\Rea \langle \nabla_i\psi,\psi\rangle e^{-f}e_i.
\end{equation}
Let $\lambda_i=\Rea \langle \nabla_i\psi,\psi\rangle e^{-f}$ so that $X=\lambda_ie_i$. Define the $(n-1)$-form 
\begin{equation}
    \alpha= \iota_X(dV_g).
\end{equation}
Then $d\alpha =\Div_g(X) dV_g$ and 
\begin{align}
    \Div_g(X)
    &=\lambda_i \Div_g(e_i)+\langle \nabla \lambda_i,e_i\rangle \\
    &=\nabla_i\lambda_i \nonumber\\
    &=\Rea \nabla_i(\langle \nabla_i\psi,\psi\rangle e^{-f}) \nonumber\\
    &= \left(\Rea \langle \nabla_i\nabla_i\psi,\psi\rangle -\Rea\langle \nabla_{\nabla f}\psi,\psi\rangle +|\nabla \psi|^2\right)e^{-f} \nonumber\\
    &=\left(\Rea \langle \dL_f\psi,\psi\rangle +|\nabla \psi|^2\right)e^{-f},\nonumber
\end{align}
hence 
\begin{equation}
    d\alpha = \left(\Rea \langle \dL_f\psi,\psi\rangle +|\nabla \psi|^2\right)e^{-f} dV_g.
\end{equation}
Stokes' theorem then gives, with $M_{\rho}=\{r\leq \rho\}\subset M$ and $S_{\rho}=\partial M_{\rho}$, that
\begin{align}\label{eqn: mass boundary term}
    \int_{M_{\rho}} \left(\Rea \langle \dL_f\psi,\psi\rangle + |\nabla  \psi|^2\right)e^{-f} dV_g 
    &=\int_{M_{\rho}}d\alpha 
    =\int_{S_{\rho}}\alpha 
    =\Rea \int_{S_{\rho}}\langle \nabla_i\psi,\psi\rangle e^{-f}\iota_{e_i}(dV_g).
\end{align}
Since $\psi=\psi_0-\xi$, the latter boundary term equals
\begin{equation}\label{eqn: expanded boundary term}
    \Rea \int_{S_{\rho}}\left(
    \langle \nabla_i\psi_0,\psi_0\rangle
    -\langle \nabla_i\xi,\psi_0\rangle
    -\langle \xi,\nabla_i\psi_0\rangle
    +\langle \nabla_i\xi,\xi\rangle \right)e^{-f}\iota_{e_i}(dV_g).
\end{equation}
Since $[e_j\cdot,e_k\cdot]$ is skew-Hermitian, as in (\ref{eqn: asymptotic expansion of Dirac operator}), it follows that
\begin{align}
    \Rea \langle \nabla_i \psi_0,\psi_0\rangle 
    &= -\frac{1}{8}\Rea (\partial_kg_{ij})\langle [e_j\cdot,e_k\cdot]\psi_0,\psi_0\rangle + O(r^{-2\tau -1}) 
    = O(r^{-2\tau -1}),
\end{align}
and so the first term in (\ref{eqn: expanded boundary term}) vanishes as $\rho\to \infty$. Also, since $\xi=O(r^{-\tau})$, $\nabla \xi=O(r^{-\tau-1})$, and $\nabla \psi_0=O(r^{-\tau-1})$, the third and fourth terms in (\ref{eqn: expanded boundary term}) also vanish as $\rho\to \infty$. Thus only the second term in (\ref{eqn: expanded boundary term}) contributes to the limit $\rho\to \infty$; the remainder of the proof consists in showing that said term equals the weighted mass.

To analyze the remaining term, let $L_i^f$ denote the operator
\begin{align}
    L_i^f
    &=\frac{1}{2}[e_i\cdot,e_j\cdot](\nabla_j -\frac{1}{2}(\nabla_jf)) \\
    &=(\delta_{ij}+e_i\cdot e_j\cdot)(\nabla_j -\frac{1}{2}(\nabla_j f) \nonumber\\
    &=\nabla_i-\frac{1}{2}(\nabla_if)+e_i\cdot D -e_i\cdot \frac{1}{2} (\nabla f)\cdot \nonumber\\
    &=\nabla_i^f+e_i\cdot D_f.\nonumber
\end{align}
If $\beta$ is the $(n-2)$-form
\begin{equation}
    \beta = e^{-f}\langle [e_i\cdot,e_j\cdot]\psi_0,\xi\rangle \iota_{e_i}\iota_{e_j} dV_g,
\end{equation}
then since $e^k\wedge \iota_{e_i}\iota_{e_j} dV_g = \delta_{ik}\iota_{e_j} dV_g-\delta_{jk}\iota_{e_i}dV_g$,
\begin{align}
    d\beta
    &=2e^{-f}((\nabla_j f)\langle [e_i\cdot,e_j\cdot]\psi_0,\xi\rangle \\
    &\qquad\qquad-\langle [e_i\cdot,e_j\cdot]\nabla_j\psi_0,\xi\rangle +\langle [e_i\cdot,e_j\cdot]\psi_0,\nabla_j\xi\rangle))\iota_{e_i}dV_g \nonumber\\
    &=-2e^{-f}(\langle [e_i\cdot,e_j\cdot](\nabla_j\psi_0-\frac{1}{2}(\nabla_jf)\psi_0),\xi\rangle  \nonumber\\
    &\qquad\qquad-\langle \psi_0,[e_i\cdot,e_j\cdot](\nabla_j\xi-\frac{1}{2}(\nabla_jf)\psi_0)\rangle)\iota_{e_i}dV_g \nonumber\\
    &=-4e^{-f}(\langle L_i^f\psi_0,\xi\rangle  -\langle \psi_0,L_i^f\xi\rangle)\iota_{e_i}dV_g. \nonumber
\end{align}
Therefore, by Stokes' theorem and the fact that $D_f\xi=D_f\psi_0$, the second term in (\ref{eqn: expanded boundary term}) is 
\begin{align}\label{eqn: expanded boundary term 2}
    -\Rea \int_{S_{\rho}}
    \langle \nabla_i\xi,\psi_0\rangle
    &e^{-f}\iota_{e_i}(dV_g) \\
    &=\Rea \int_{S_{\rho}}
    \langle e_i\cdot D_f\xi-L_i^f\xi-\frac{1}{2}(\nabla_i f)\xi,\psi_0\rangle
    e^{-f}\iota_{e_i}(dV_g) \nonumber\\
    &=\Rea \int_{S_{\rho}}
    \left(\langle e_i\cdot D_f\psi_0,\psi_0\rangle-\langle \xi,L_i^f\psi_0\rangle-\frac{1}{2}\langle (\nabla_i f)\xi,\psi_0\rangle\right)
    e^{-f}\iota_{e_i}(dV_g). \nonumber
\end{align}
Since $f\to 0$ at infinity, $\nabla f= O(r^{\delta-1})$, where $\delta-1<\tau-(n-1)$ by (\ref{eqn: positive weighted mass assumptions}), and $\xi=O(r^{-\tau})$, the last term above vanishes as $\rho\to \infty$. 
Similarly, the second term above vanishes in the limit. On the other hand, (\ref{eqn: asymptotic expansion of Dirac operator}) gives
\begin{align}
    e_i\cdot D_f \psi_0
    &=-\frac{1}{8}(\partial_kg_{lj})e_i\cdot e_l\cdot [e_j\cdot,e_k\cdot]\psi_0-\frac{1}{2}e_i\cdot (\nabla f)\cdot \psi_0+ O(r^{-2\tau-1})\psi_0 \\
    &=-\frac{1}{4}(\partial_kg_{lj})e_i\cdot e_l\cdot (\delta_{jk}+e_j\cdot e_k\cdot)\psi_0-\frac{1}{2}e_i\cdot (\nabla f)\cdot \psi_0 +O(r^{-2\tau-1})\psi_0 \nonumber\\
    &=-\frac{1}{4}(\partial_jg_{kj}-\partial_k g_{jj})e_i\cdot e_k\cdot\psi_0-\frac{1}{2}e_i\cdot (\nabla f)\cdot \psi_0 +O(r^{-2\tau-1})\psi_0. \nonumber
\end{align}
Writing $e_i\cdot e_k\cdot = \frac{1}{2}[e_i\cdot, e_k\cdot] -\delta_{ik}$ and noting that $[e_i\cdot,e_k\cdot]$ is skew, it follows that 
\begin{equation}
    \Rea \langle e_i\cdot D_f\psi_0,\psi_0\rangle 
    =\frac{1}{4}(\partial_jg_{ij}-\partial_ig_{jj}+2(\nabla_i f)+O(r^{-2\tau-1}))|\psi_0|^2.
\end{equation}
and hence (\ref{eqn: expanded boundary term 2}) becomes
\begin{equation}
    \frac{1}{4}\int_{S_{\rho}}\left(\partial_jg_{ij}-\partial_ig_{jj}+2(\nabla_i f)+O(r^{-2\tau-1})\right)|\psi_0|^2e^{-f}\iota_{e_i} dV_g.
\end{equation}
Putting this into (\ref{eqn: mass boundary term}), letting $\rho\to \infty$ and using the definition of mass (\ref{eqn: definition of mass}) gives the formula
\begin{equation}
    \int_M\left(|\nabla \psi|^2+\frac{1}{4}\Scal_f|\psi|^2\right)e^{-f} dV_g 
    = \frac{1}{4}\mass(g)+\frac{1}{2}\lim_{\rho\to \infty}\int_{S_{\rho}}|\psi_0|^2e^{-f}\iota_{\nabla f}dV_g.
\end{equation}
Finally, a coordinate calculation shows that 
\begin{equation}
    \int_{S_{\rho}}\langle \nabla f,\nu\rangle\,e^{-f} dA
    =\int_{S_{\rho}}e^{-f}\iota_{\nabla f}dV_g,
\end{equation}
and since $|\psi_0|\to 1$ at infinity, the second-to-last equation gives the weighted Witten formula
\begin{equation}
    \int_M\left(|\nabla \psi|^2+\frac{1}{4}\Scal_f|\psi|^2\right)e^{-f} dV_g 
    = \frac{1}{4}\mass(g)+\frac{1}{2}\lim_{\rho\to \infty}\int_{S_{\rho}}\langle \nabla f,\nu\rangle \,e^{-f}dA.
\end{equation}
\end{proof}

\newcommand{\Addresses}{{
  \bigskip
  \bigskip
  \small
    \textsc{Department of Mathematics}\par\nopagebreak
    \textsc{Massachusetts Institute of Technology}\par\nopagebreak
    \textsc{77 Massachusetts Avenue} \par\nopagebreak
    \textsc{Cambridge, MA 02139} 
    
    \medskip
    \medskip
    
    \textit{Correspondence to be sent to:} \texttt{juliusbl@mit.edu}

}}

\Addresses

\end{document}